\documentclass[smallextended]{svjour3}       
\smartqed  
\usepackage{amsmath,amssymb,amsfonts}
\usepackage{setspace,cite,indentfirst}
\usepackage{array}
\usepackage [latin1]{inputenc}
\usepackage{graphicx}
\usepackage[caption=false,font=footnotesize]{subfig}
\usepackage{tabu}
\usepackage{booktabs}
\usepackage{color}
\usepackage{setspace}
\DeclareGraphicsExtensions{.eps}

\begin{document}

\title{Uncertainty principle for the two-sided quaternion
windowed linear canonical transform
}


\author{Wen-Biao Gao       \and
        Bing-Zhao Li* 
}


\institute{Wen-Biao Gao \at
           School of Mathematics and Statistics, Beijing Institute of Technology, Beijing 102488, P.R. China\\
             \email{wenbiaogao@163.com}          
           \and
           Corresponding author, Dr. Bing-Zhao Li \at
           Tel.: 86-10-81384471,  Fax: 86-10-81384701\\
              Beijing Key Laboratory on MCAACI, Beijing Institute of Technology, Beijing 102488, P.R. China\\
               \email{li$\_$bingzhao@bit.edu.cn}
}

\date{Received: date / Accepted: date}

\maketitle

\begin{abstract}
In this paper, we investigate the (two-sided) quaternion windowed linear canonical transform (QWLCT) and study the uncertainty principles associated with the QWLCT.
Firstly, several important properties of the QWLCT such as bounded, shift, modulation, orthogonality relation, are presented based on the spectral representation of the quaternionic linear canonical transform (QLCT). Secondly, Pitt's inequality and Lieb inequality for the QWLCT are explored.
Moreover, we study different kinds of uncertainty principles for the QWLCT, such as Logarithmic uncertainty principle, Entropic uncertainty principle, Lieb uncertainty principle and Donoho-Stark's uncertainty principle.
Finally, we give a numerical example and a potential application in signal recovery by using Donoho-Stark's uncertainty principle associated with the QWLCT.
\end{abstract}

\section{Introduction}
\label{intro}
The linear canonical transform (LCT) plays an important role for chirp signal
analysis in parameter estimation, applied mathematics, signal
processing, radar system analysis, filter design, phase retrieval, and pattern recognition and optics\cite{7,8,9,10,11,12}. However, the LCT can't show the local LCT-frequency contents as a result of its global kernel. To meet this request, the LCT has been successfully applied to research the generalized windowed function in \cite{15,29}. The windowed linear canonical transform (WLCT) is a method devised to study signals whose spectral content changes
in time. As shown in \cite{14,8,51,52} some important properties of the WLCT such as the analogue of the Poisson summation formula, sampling formulas, Paley-Wiener theorem, and dual window solution
are discussed. The discrete WLCT is discussed in \cite{50}. It presents the time and LCT-frequency information, and is originally a local LCT distribution.

In the last few years, some authors have generalized the LCT to quaternion-valued signals, known as the quaternionic linear canonical
transform (QLCT). It transforms a quaternionic 2D signal into a quaternion
valued frequency domain signal which is an effective processing tool for color
image analysis.
 The QLCT was firstly studied in \cite{17} including prolate spheroidal wave signals and uncertainty principles \cite{19}. Some useful properties of
the QLCT such as linearity, reconstruction formula, continuity, boundedness, positivity inversion formula and the uncertainty
principle were established in \cite{17,13,18,20,21,30}. An application of the QLCT to study of generalized
swept-frequency filters was introduced in \cite{23}.
According to the QLCT, Xiong and Fu \cite{XF2015} have proposed the two-side quaternion linear canonical transform
associated with the real window function-called the quaternion window linear canonical transform (QWLCT),
and they studied several
properties and Heisenberg's uncertainty principle for the QWLCT \cite{XF2015}. Different from the \cite{XF2015},
based on the quaternion window function, we \cite{37} have investigated some properties of the two-side QWLCT, and gave an example to illustrative the QWLCT is highly effective.

On the other hand, the uncertainty principle plays a vital role in time-frequency signal analysis and quantun mechanics \cite{1,2,3,4,5,6}.
It was first proposed by the German physicist W. Heisenberg in 1927 \cite{6}.
In quantum mechanics, an uncertainty principle
asserts that one cannot make certain of the position and velocity of an electron
(or any particle) at the same time. That is, increasing the knowledge
of the position decreases the knowledge of the velocity or momentum of an
electron. In signal processing, an uncertainty principle states that the product
of the variances of the signal in the time and frequency domains has a lower bound.
There are many different kinds of uncertainty principles, for instance, Heisenberg uncertainty principle \cite{37,47,53,54},
 Logarithmic uncertainty principle \cite{36}, Hardy's uncertainty principle \cite{47,48}, Beurling's uncertainty principle \cite{49},
 Lieb uncertainty principle \cite{15,40}, Donoho-Stark's uncertainty principle \cite{42,45,46}.
 The Lieb uncertainty principle for the WLCT has been discussed
in \cite{15}, which takes the LCT version as one of its special cases.
 In \cite{53}, Xu et al. studied the uncertainty relation for the WLCT in two WLCT domains. Huang
et al. \cite{54} discussed the uncertainty principle and orthogonal condition for the
WLCT in one WLCT domains.
In \cite{XF2015}, based on
the real window function, the Heisenberg's uncertainty principle for the QWLCT has been obtained.

However, to the best of our knowledge, other kinds of the uncertainty principles for the QWLCT have not been found in the literature.
The aim of this paper is to obtain several uncertainty principles for the QWLCT.
First of all, some properties of the QWLCT are reviewed.
Second, based on the relationship between the QWLCT and the quaternion Fourier transform (QFT) \cite{47,48,49,51,52}, we establish the Pitt's inequality and Lieb inequality associated with the QWLCT.
Third, the uncertainty principles for the QWLCT such as Logarithmic uncertainty principle, Entropic uncertainty principle,
Lieb uncertainty principle and Donoho-Stark's uncertainty principle are obtained.
Finally, a numerical example and a potential application in signal recovery associated with Donoho-Stark's uncertainty principle are given.

The paper is organized as follows: Section 2 gives a brief introduction to some general
definitions and basic properties of quaternion algebra, QLCTs of 2D quaternion-valued signals.
We present the definition and the properties of the QWLCT in Section 3.
In Section 4, provides the uncertainty principles associated with the QWLCT. We give a numerical example and a potential application in Section 5.
In Section 6, some conclusions are drawn.
\section{Preliminaries}
\label{sec:1}
In this section, we mainly review some basic facts on the quaternion algebra and the QLCT, which will be needed throughout the paper.
\subsection{Quaternion algebra}
The quaternion algebra is an extension of the complex number to 4D algebra. It was first invented by W. R. Hamilton in 1843 and classically denoted by $\mathbb{H}$ in his honor. Every element of $\mathbb{H}$ has a Cartesian form given by
\begin{align}
\mathbb{H}=\left\lbrace q|q:=[q]_0+\mathbf{i}[q]_1+\mathbf{j}[q]_2+\mathbf{k}[q]_3, [q]_i\in\mathbb{R}, i=0,1,2,3\right\rbrace
\end{align}
where $\mathbf{i}, \mathbf{j}, \mathbf{k}$ are imaginary units obeying Hamilton's multiplication rules (see \cite{16})
\begin{align}
\begin{split}
		&\mathbf{i}^2=\mathbf{j}^2=\mathbf{k}^2=-1,\\
&\mathbf{i}\mathbf{j}=-\mathbf{j}\mathbf{i}=\mathbf{k},\,\mathbf{j}\mathbf{k}=-\mathbf{k}\mathbf{j}=\mathbf{i},\,\mathbf{k}\mathbf{i}
=-\mathbf{i}\mathbf{k}=\mathbf{j}.
\end{split}
\end{align}
Let $[q]_{0}$ and $q=\mathbf{i}[q]_1+\mathbf{j}[q]_2+\mathbf{k}[q]_3$ denote the real scalar part and the vector part of quaternion number $q=[q]_0+\mathbf{i}[q]_1+\mathbf{j}[q]_2+\mathbf{k}[q]_3$, respectively. Then, from \cite{53}, the
real scalar part has a cyclic multiplication symmetry
\begin{align}
\begin{split}
		[pql]_{0}=[qlp]_{0}=[lpq]_{0}, \ \  \forall q,p,l\in \mathbb{H},
\end{split}
\end{align}
the conjugate of a quaternion $q$ is defined by $\overline{q}=[q]_0-\mathbf{i}[q]_1-\mathbf{j}[q]_2-\mathbf{k}[q]_3$, and the norm of $q\in \mathbb{H}$ defined as
\begin{align}
\left| q\right|=\sqrt{q\bar{q}}=\sqrt{[q]_0^2+[q]_1^2+[q]_2^2+[q]_3^2}\textcolor{red}{ .}
\end{align}
It is easy to verify that
\begin{align}
\begin{split}
		\overline{pq}=\overline{q} \ \overline{p}, |qp|=|q||p|, \ \ \forall q,p\in \mathbb{H}.
\end{split}
\end{align}
If $1\leq s<\infty$, the quaternion modules $L^s(\mathbb{R}^t,\mathbb{H})$ are defined as
\begin{align}\label{qm}
\begin{split}
	L^s(\mathbb{R}^t,\mathbb{H}):=\left\{f|f:\mathbb{R}^t\rightarrow \mathbb{H}, \|f\|_{L^s(\mathbb{R}^t,\mathbb{H})}
=\left(\int_{\mathbb{R}^t}|f(\mathbf{x})|^{s}{\rm d}\mathbf{x}\right)^{\frac{1}{s}}<\infty\right\}.	
\end{split}
\end{align}
For $s=\infty$, $L^\infty(\mathbb{R}^t,\mathbb{H})$ is a collection of essentially bounded measurable functions with the norm
\begin{align*}
\begin{split}
	 \|f\|_{L^\infty(\mathbb{R}^t,\mathbb{H})}=ess\ \ sup|f(\mathbf{x})|,	
\end{split}
\end{align*}
if $f\in L^\infty(\mathbb{R}^t,\mathbb{H})$ is continuous then
\begin{align*}
\begin{split}
	 \|f\|_{L^\infty(\mathbb{R}^t,\mathbb{H})}=sup|f(\mathbf{x})|.
\end{split}
\end{align*}

Now we introduce an inner product of quaternion functions $f,g$ defined on $L^2(\mathbb{R}^2,\mathbb{H})$ given by
\begin{align}
\label{inner}
	( f,g)_{L^2(\mathbb{R}^2,\mathbb{H})}=\int_{\mathbb{R}^2}f(\mathbf{x})\overline{g(\mathbf{x})}{\rm d}\mathbf{x},
\end{align}
with symmetric real scalar part
\begin{align}
	\langle f,g\rangle=\frac{1}{2}\left \{( f,g)+( g,f)\right\} =\int_{\mathbb{R}^2}[f(\mathbf{x})\overline{g(\mathbf{x})}]_{0}{\rm d}\mathbf{x},
\end{align}
where $\mathbf{x}=(x_{1},x_{2})\in \mathbb{R}^{2}$ and $ {\rm d}\mathbf{x}={\rm d}x_1{\rm d}x_2$.

The associated scalar norm of $f(\mathbf{x})\in L^2(\mathbb{R}^2,\mathbb{H})$ is defined by both (7) and (8):
\begin{align}
	\|f\|_{L^2(\mathbb{R}^2,\mathbb{H})}^2=\langle f,f\rangle_{L^2(\mathbb{R}^2,\mathbb{H})}= \int_{\mathbb{R}^2}|f(\mathbf{x})|^2 {\rm d}\mathbf{x}<\infty.
\end{align}
The convolution of $f\in L^2(\mathbb{R}^t,\mathbb{H})$ and $g\in L^2(\mathbb{R}^t,\mathbb{H})$, denoted by $f*g$, is defined by
\begin{align*}
	(f*g)(\mathbf{y})=\int_{\mathbb{R}^t}f(\mathbf{x})g(\mathbf{y-x})\rm{d}\mathbf{x}.
\end{align*}
where $\mathbf{y}=(y_{1},y_{2})\in \mathbb{R}^{2}$.
\begin{lemma}
If $f,g\in L^2(\mathbb{R}^2,\mathbb{H})$, then the Cauchy-Schwarz inequality holds[25]
\begin{align}
	\left| \langle f,g\rangle_{L^2(\mathbb{R}^2,\mathbb{H})}\right|^2\leq \|f\|_{L^2(\mathbb{R}^2,\mathbb{H})}^2\|g\|_{L^2(\mathbb{R}^2,\mathbb{H})} ^2.
\end{align}
If and only if $f=\lambda g$ for some quaternionic parameter $\lambda \in\mathbb{H}$, the equality holds. \end{lemma}
\subsection{The quaternion linear canonical transform}
The QLCT is firstly defined by Kou, et., which  is a generalization of the LCT in the frame of quaternion algebra \cite{50,16}. Due to the non-commutativity of quaternion multiplication, there are three different types of the QLCT: the left-sided QLCT, the right-sided QLCT, and the two-sided QLCT. In this paper, we mainly focus on the two-sided QLCT.
\begin{definition} Let
	$A_i=\begin{bmatrix}
	a_i&b_i\\
	c_i&d_i
	\end{bmatrix}\in \mathbb{R}^{2\times 2}$ be a matrix parameter satisfying ${\rm det}(A_i)=1$, for $i=1,2$. The two-sided QLCT of signal $f\in L^2\left( \mathbb{R}^2,\mathbb{H}\right)$ is defined by
	\begin{align}
	\label{dQLCT}
	\mathcal{L}_{A_1,A_2}\{f\}(\mathbf{w})=
\int_{\mathbb{R}^2}K_{A_1}^{\mathbf{i}}(x_1,\omega_1)f(\mathbf{x})K_{A_2}^{\mathbf{j}}(x_2,\omega_2)\rm{d}\mathbf{x},
	\end{align}	
	where $\mathbf{w}=(\omega_1,\omega_2)\in \mathbb{R}^{2}$ is regarded as the QLCT domain, and  the kernel signals $K_{A_1}^{\mathbf{i}}(x_1,\omega_1)$, $K_{A_2}^{\mathbf{j}}(x_2,\omega_2)$ are respectively given by
	\begin{align}
		\begin{split}
			K_{A_1}^{\mathbf{i}}(x_1,\omega_1):=\begin{cases}
		\frac{1}{\sqrt{2\pi b_1}}e^{\mathbf{i}\left(\frac{a_1}{2b_1}x_1^2-\frac{x_1\omega_1}{b_1}+\frac{d_1}{2b_1}\omega_1^2-\frac{\pi}{4} \right) },   &b_1\neq0  \\
		\sqrt{d_1}e^{\mathbf{i}\frac{c_1d_1}{2}\omega_1^2}\delta(x_{1}-d_{1}w_{1}),    &b_1=0
		\end{cases}
		\end{split}
	\end{align}
	and
	\begin{align}
		\begin{split}
			K_{A_2}^{\mathbf{j}}(x_2,\omega_2):=\begin{cases}
		\frac{1}{\sqrt{2\pi b_2}}e^{\mathbf{j}\left(\frac{a_2}{2b_2}x_2^2-\frac{x_2\omega_2}{b_2}+\frac{d_2}{2b_2}\omega_2^2-\frac{\pi}{4} \right) },   &b_2\neq0  \\
		\sqrt{d_2}e^{\mathbf{j}\frac{c_2d_2}{2}\omega_2^2}\delta(x_{2}-d_{2}w_{2}),    &b_2=0
		\end{cases}
		\end{split}
	\end{align}
where $\delta(x)$ representing the Dirac function.
\end{definition}
From the above definition, it is noted that for $b_i=0,i=1,2$ the QLCT of a signal is a kind of scaling and chirp multiplication operations, and it is of no particular interest for our objective in this work. Hence, without loss of generality, we set $b_i\neq0$ in the following section unless stated otherwise. Under some suitable conditions, the QLCT above is invertible and the inversion is given in the following section.
\begin{lemma} Suppose $f\in L^2\left( \mathbb{R}^2,\mathbb{H}\right)$, then the inversion of the QLCT of $f$ is given by
	\begin{align}
		\begin{split}
		\label{eIQLCT}
		f(\mathbf{x})&=\mathcal{L}_{A_1,A_2}^{-1}[\mathcal{L}_{A_1,A_2}\{f\}](\mathbf{x})\\
&=\int_{\mathbb{R}^2}K_{A_1}^{-\mathbf{i}}(x_1,\omega_1)
\mathcal{L}_{A_1,A_2}\left\{f\right\}(\mathbf{w})K_{A_2}^{-\mathbf{j}}(x_2,\omega_2)\rm{d}\mathbf{w}.
		\end{split}
	\end{align}	
\end{lemma}
\section{Quaternionic windowed linear canonical transform (QWLCT)}
In this section, the generalization of window function associated with the QLCT will be discussed, which is denoted by QWLCT. Moreover, several basic
properties of them are presented.
\subsection{The definition of 2D QWLCT}
This section leads to 2D quaternion window function associated with the QLCT. Due to the noncommutative property of
multiplication of quaternions, there are three different types of
QWLCTs: two-sided QWLCT, left-sided QWLCT and right-sided QWLCT. Alternatively, we use the two-sided QWLCT to define the QWLCT.
\begin{definition}[QWLCT] \cite{XF2015,37} Let
	$A_i=\begin{bmatrix}
	a_i&b_i\\
	c_i&d_i
	\end{bmatrix}\in \mathbb{R}^{2\times 2}$ be a matrix parameter satisfying ${\rm det}(A_i)=1$, for $i=1,2$. Let $\phi\in L^2(\mathbb{R}^2,\mathbb{H})\backslash \{0\}$ be a quaternion window function. The two-sided QWLCT of a signal $f\in L^2\left( \mathbb{R}^2,\mathbb{H}\right)$ with respect to $\phi$ is defined by
	\begin{align}
	\label{dQLCT}
	G^{A_{1},A_{2}}_{\phi}\{f\}(\mathbf{w,u})&=\int_{\mathbb{R}^2}K_{A_1}^{\mathbf{i}}(x_1,\omega_1)f(\mathbf{x}) \overline{\phi(\mathbf{x-u})} K_{A_2}^{\mathbf{j}}(x_2,\omega_2)\rm{d}\mathbf{x},
	\end{align}	
where $\mathbf{u}=(u_1,u_2)\in\mathbb{R}^2$, $K_{A_1}^{\mathbf{i}}(x_1,\omega_1)$ and $K_{A_2}^{\mathbf{j}}(x_2,\omega_2)$ are given by (12) and (13), respectively.
\end{definition}
For a fixed $\mathbf{u}$, we have
\begin{align}
		\begin{split}
		G^{A_{1},A_{2}}_{\phi}\{f\}(\mathbf{w,u})=\mathcal{L}_{A_1,A_2}{\{f(\mathbf{x})
\overline{\phi(\mathbf{x-u})}\}(\mathbf{w})}. \\
		\end{split}
	\end{align}
Applying the inverse QLCT to (16), we have
\begin{align}\label{hn}
		\begin{split}
		f_{\mathbf{u}}(\mathbf{x})=f(\mathbf{x})\overline{\phi(\mathbf{x-u})}&=
\mathcal{L}_{A_1,A_2}^{-1}\{G^{A_1,A_2}_{\phi}\{f\}(\mathbf{w,u})\} \\
&=\int_{\mathbb{R}^2}K_{A_1}^{-\mathbf{i}(x_1,\omega_1)}
G^{A_1,A_2}_{\phi}\{f\}(\mathbf{w,u})K_{A_2}^{-\mathbf{j}(x_2,\omega_2)} \rm{d}\mathbf{w},
		\end{split}
	\end{align}
where $f_{\mathbf{u}}(\mathbf{x})$ is called modified signal.

Moreover, we can also obtain the relationship between the QWLCT and the QFT [38]:
\begin{align}\label{WF}
		\begin{split}
		G^{A_{1},A_{2}}_{\phi}\{f\}(\mathbf{w,u})=\frac{1}{\sqrt{2\pi b_{1}}}e^{\mathbf{i}\left(\frac{d_1}{2b_1}\omega_1^2-\frac{\pi}{4} \right)}F_{Q}(h)\left(\frac{\mathbf{w}}{\mathbf{b}},\mathbf{u}\right)\frac{1}{\sqrt{2\pi b_2}}e^{\mathbf{j}\left(\frac{d_2}{2b_2}\omega_2^2-\frac{\pi}{4} \right)},
		\end{split}
	\end{align}
where $\mathbf{b}=(b_1,b_2)\in\mathbb{R}^2$, $F_{Q}(f)(\mathbf{w})=\int_{\mathbb{R}^2}e^{-\mathbf{i}x_1\omega_1}f(\mathbf{x})e^{-\mathbf{j}x_2\omega_2}\rm{d}\mathbf{x}$ is the QFT of the signal $f(\mathbf{x})$ and
\begin{align}\label{hf}
		\begin{split}
		h(\mathbf{x,u})&=e^{\mathbf{i}\frac{a_1}{2b_1}x_1^2}f(\mathbf{x})\overline{\phi(\mathbf{x-u})}
e^{\mathbf{j}\frac{a_2}{2b_2}x_2^2}\\
&=e^{\mathbf{i}\frac{a_1}{2b_1}x_1^2}f_{\mathbf{u}}(\mathbf{x})
e^{\mathbf{j}\frac{a_2}{2b_2}x_2^2}.
		\end{split}
	\end{align}
\subsection{Some properties of QWLCT}
In this subsection, we present several basic properties of the QWLCT. These properties have been proved in \cite{XF2015,37} .
\begin{property}[Boundedness] Let $\phi\in L^2(\mathbb{R}^2,\mathbb{H})\backslash \{0\}$ be a window function and $f\in L^2\left( \mathbb{R}^2,\mathbb{H}\right)$, then
\begin{align}
		\begin{split}
		|G^{A_1,A_2}_{\phi}\{f\}(\mathbf{w,u})|\leq \frac{1}{2\pi\sqrt{|b_{1}b_{2}|}} \|f\|_{L^{2}(\mathbb{R}^2,\mathbb{H})} \|\phi\|_{L^{2}(\mathbb{R}^2,\mathbb{H})}.
		\end{split}
	\end{align}
Furthermore, we have
\begin{align} \label{wq}
		\begin{split}
		\|G^{A_1,A_2}_{\phi}\{f\}(\mathbf{w,u})\|_{L^\infty(\mathbb{R}^4,\mathbb{H})}\leq \frac{1}{2\pi\sqrt{|b_{1}b_{2}|}} \|f\|_{L^{2}(\mathbb{R}^2,\mathbb{H})} \|\phi\|_{L^{2}(\mathbb{R}^2,\mathbb{H})}.
		\end{split}
	\end{align}
\end{property}
\begin{property}[Linearity] Let $\phi\in L^2(\mathbb{R}^2,\mathbb{H})\backslash \{0\}$ be a window function and $f,g \in L^2\left( \mathbb{R}^2,\mathbb{H}\right)$ . The QWLCT is a linear operator, namely,
  \begin{align}
        \begin{split}
		[G^{A_1,A_2}_{\phi}\{\lambda f+\mu g\}](\mathbf{w,u})= \lambda G^{A_1,A_2}_{\phi}\{f\}(\mathbf{w,u})+\mu G^{A_1,A_2}_{\phi}\{g\}(\mathbf{w,u}),
		\end{split}
	\end{align}
for arbitrary real constants $\lambda$ and $\mu$ . \end{property}
\begin{property}[Parity]
Let $\phi\in L^2(\mathbb{R}^2,\mathbb{H})\backslash \{0\}$ be a window function and $f \in L^2\left( \mathbb{R}^2,\mathbb{H}\right)$. Then we have
\begin{align}
        \begin{split}
		G^{A_1,A_2}_{P\phi}\{P f\}(\mathbf{w,u})= G^{A_1,A_2}_{\phi}\{f\}(\mathbf{-w,-u}),
		\end{split}
	\end{align}
where $P\phi(\mathbf{x})=\phi(-\mathbf{x})$ for every window function $\phi\in L^2(\mathbb{R}^2,\mathbb{H})$. \end{property}
\begin{property}[Shift]
 Let $\phi\in L^2(\mathbb{R}^2,\mathbb{H})\backslash \{0\}$ be a window function and $f \in L^2\left( \mathbb{R}^2,\mathbb{H}\right)$. Then we have
\begin{align}
        \begin{split}
		G^{A_1,A_2}_{\phi}\{T_{\mathbf{r}}f\}(\mathbf{w,u})= e^{\mathbf{i}r_{1}\omega_{1}c_{1}}e^{-\mathbf{i}\frac{a_{1}r_{1}^{2}}{2}c_{1}}
G^{A_1,A_2}_{\phi}\{f\}(\mathbf{m,n}) e^{\mathbf{j}r_{2}\omega_{2}c_{2}}e^{-\mathbf{j}\frac{a_{2}r_{2}^{2}}{2}c_{2}},
		\end{split}
	\end{align}
where $T_{\mathbf{r}}f(\mathbf{x})=f(\mathbf{x-r}), \mathbf{r}=(r_{1},r_{2})$,
 $\mathbf{m}=(m_{1},m_{2})$, $ \mathbf{n}=(n_{1},n_{2})\in \mathbb{R}^2$, $ m_{i}=w_{i}-a_{i}r_{i} $, $ n_{i}=u_{i}-r_{i} $, $i=1,2$.\end{property}
\begin{property}[Modulation]
Let $\phi\in L^2(\mathbb{R}^2,\mathbb{H})\backslash \{0\}$ be a window function and $f \in L^2\left( \mathbb{R}^2,\mathbb{H}\right)$.
$\mathbb{M}_{s}f$ be modulation operator defined by $\mathbb{M}_{s}f(\mathbf{x})=e^{\mathbf{i}x_{1}s_{1}}f(\mathbf{x})e^{\mathbf{j}x_{2}s_{2}}$ with $\mathbf{s}=(s_{1},s_{2})\in \mathbb{R}^2$.
Then we have
\begin{align}
        \begin{split}
		G^{A_1,A_2}_{\phi}\{\mathbb{M}_{s}f\}(\mathbf{w,u})&=e^{\mathbf{i}\omega_{1}s_{1}d_{1}}e^{-\mathbf{i}\frac{b_{1}d_{1}s_{1}^{2}}{2}}
G^{A_1,A_2}_{\phi}\{f\}(\mathbf{v,u})e^{\mathbf{j}\omega_{2}s_{2}d_{2}}e^{-\mathbf{j}\frac{b_{2}d_{2}s_{2}^{2}}{2}},
\end{split}
	\end{align}
where $\mathbf{v}=(v_{1},v_{2})\in \mathbb{R}^2$,
$ v_{i}=w_{i}-s_{i}b_{i} $, $i=1,2$.
\end{property}
\begin{property}[Inversion formula]
 Let $\phi\in L^2(\mathbb{R}^2,\mathbb{H})\backslash \{0\}$ be a window function, $ 0<\|\phi\|^{2}<\infty $ and $f \in L^2\left( \mathbb{R}^2,\mathbb{H}\right)$.
Then we have the inversion formula of the QWLCT,
\begin{align}
        \begin{split}
		f(\mathbf{x})=\frac{1}{\|\phi\|^{2}}\int_{\mathbb{R}^2}\int_{\mathbb{R}^2}K_{A_1}^{-\mathbf{i}(x_1,\omega_1)}
G^{A_1,A_2}_{\phi}\{f\}(\mathbf{w,u})K_{A_2}^{-\mathbf{j}(x_2,\omega_2)} \phi\mathbf{(x-u)}\rm{d}\mathbf{w}\rm{d}\mathbf{u}.
\end{split}
	\end{align} \end{property}
\begin{property}[Parseval's theorem]
 Let $\phi,\psi \in L^2(\mathbb{R}^2,\mathbb{H})\backslash \{0\}$ be window functions and $f,g \in L^2\left( \mathbb{R}^2,\mathbb{H}\right)$.
Then
\begin{align}
		\begin{split}
		\langle G^{A_1,A_2}_{\phi}\{f\}(\mathbf{w,u}), {G^{A_1,A_2}_{\psi}\{g\}(\mathbf{w,u})} \rangle
=[( f,g)(\phi,\psi)]_{0}.
		\end{split}
	\end{align} \end{property}
Based on the above theorem, we may conclude the following important consequences.

(i) If $\phi=\psi$, then
\begin{align}
		\begin{split}
		\langle G^{A_1,A_2}_{\phi}\{f\}(\mathbf{w,u}), {G^{A_1,A_2}_{\phi}\{g\}(\mathbf{w,u})} \rangle
=\|\phi\|^{2}_{L^2(\mathbb{R}^2,\mathbb{H})}\langle f,g\rangle.
		\end{split}
	\end{align}

(ii) If $f=g$, then
\begin{align}
		\begin{split}
		\langle G^{A_1,A_2}_{\phi}\{f\}(\mathbf{w,u}), {G^{A_1,A_2}_{\psi}\{f\}(\mathbf{w,u})} \rangle
=\|f\|^{2}_{L^2(\mathbb{R}^2,\mathbb{H})}\langle \phi,\psi\rangle.
		\end{split}
	\end{align}

(iii) If  $f=g$ and $\phi=\psi$, then
\begin{align}\label{fsxd}
		\begin{split}
		\langle G^{A_1,A_2}_{\phi}\{f\}(\mathbf{w,u}), {G^{A_1,A_2}_{\phi}\{f\}(\mathbf{w,u})} \rangle&=
\int_{\mathbb{R}^2}\int_{\mathbb{R}^2}| G^{A_1,A_2}_{\phi}\{f\}(\mathbf{w,u})|^{2}\rm{d}\mathbf{w}\rm{d}\mathbf{u}\\
&=\|f\|^{2}_{L^2(\mathbb{R}^2,\mathbb{H})} \|\phi\|^{2}_{L^2(\mathbb{R}^2,\mathbb{H})}.
		\end{split}
	\end{align}
\subsection{Pitt's inequality}
\begin{lemma} (Pitt's inequality of the QFT \cite{36})\label{PQFT}
For $f\in S(\mathbb{R}^t,\mathbb{H})$,
\begin{align}
		\begin{split}
		\int_{\mathbb{R}^t}|\mathbf{w}|^{-\alpha}|F_{Q}(f)(\mathbf{w})|^{2}\rm{d}\mathbf{w}\leq M_{\alpha}\int_{\mathbb{R}^t}|\mathbf{x}|^{\alpha}|\textit{f}(\mathbf{x})|^{2}\rm{d}\mathbf{x},
		\end{split}
	\end{align}
where $M_{\alpha}=\pi^{\alpha}\left[\Gamma\left( \frac{2t-\alpha}{4}\right)\Gamma\left( \frac{2t+\alpha}{4}\right) \right]$, $0\leq \alpha\leq t$ and
$S(\mathbb{R}^t,\mathbb{H})$ denotes the Schwartz class and  $\Gamma$ is Gamma function.
\end{lemma}
According to the above Lemma, we obtain the Pitt's inequality of the QWLCT.
\begin{theorem} (Pitt's inequality of the QWLCT)
For $f\in S(\mathbb{R}^t,\mathbb{H})$,
\begin{align}\label{pi}
		\begin{split}
		&\int_{\mathbb{R}^t}\int_{\mathbb{R}^t}|\mathbf{w}|^{-\alpha}|G^{A_1,A_2}_{\phi}\{f\}(\mathbf{w,u})|^{2}\rm{d}\mathbf{u}\rm{d}\mathbf{w}\\
&\leq \frac{1}{4\pi^{2}|\mathbf{b}|^{\alpha}}M_{\alpha}\|\phi\|^{2}_{L^2(\mathbb{R}^t,\mathbb{H})}
\int_{\mathbb{R}^t}|\mathbf{x}|^{\alpha}|\textit{f}(\mathbf{x})|^{2}\rm{d}\mathbf{x},
		\end{split}
	\end{align}
where $M_{\alpha}=\pi^{\alpha}\left[\Gamma\left( \frac{2t-\alpha}{4}\right)\Gamma\left( \frac{2t+\alpha}{4}\right) \right]$, $0\leq \alpha\leq t$ and $\Gamma$ is Gamma function.
\end{theorem}
\begin{proof}
From (18), we have
\begin{align*}
		\begin{split}
		&\int_{\mathbb{R}^t}\int_{\mathbb{R}^t}|\mathbf{w}|^{-\alpha}|G^{A_1,A_2}_{\phi}\{f\}(\mathbf{w,u})|^{2}\rm{d}\mathbf{u}\rm{d}\mathbf{w}\\&=
\frac{1}{4\pi^{2}|\mathbf{b}|}\int_{\mathbb{R}^t}\int_{\mathbb{R}^t}|\mathbf{w}|^{-\alpha}|\textit{F}_{Q}(\textit{h})\left(\mathbf{\frac{w}{b},u}\right)|^{2}\rm{d}\mathbf{u}\rm{d}\mathbf{w}.
		\end{split}
	\end{align*}
Let $\mathbf{\frac{w}{b}}=\mathbf{\xi}$, then
\begin{align*}
		\begin{split}
		&\int_{\mathbb{R}^t}\int_{\mathbb{R}^t}|\mathbf{w}|^{-\alpha}|G^{A_1,A_2}_{\phi}\{f\}(\mathbf{w,u})|^{2}\rm{d}\mathbf{u}\rm{d}\mathbf{w}\\&=
\frac{1}{4\pi^{2}|\mathbf{b}|^{\alpha}}\int_{\mathbb{R}^t}\int_{\mathbb{R}^t}|\mathbf{\xi}|^{-\alpha}
|F_{Q}(\textit{h})\left(\mathbf{\xi,u}\right)|^{2}\rm{d}\mathbf{u}\rm{d}\mathbf{\xi}.
		\end{split}
	\end{align*}
By Lemma \ref{PQFT}, we obtain
\begin{align*}
		\begin{split}
		&\int_{\mathbb{R}^t}\int_{\mathbb{R}^t}|\mathbf{w}|^{-\alpha}|G^{A_1,A_2}_{\phi}\{f\}(\mathbf{w,u})|^{2}\rm{d}\mathbf{u}\rm{d}\mathbf{w}\\
&\leq \frac{1}{4\pi^{2}|\mathbf{b}|^{\alpha}}M_{\alpha}\int_{\mathbb{R}^t}\int_{\mathbb{R}^t}|\mathbf{x}|^{\alpha}|\textit{h}(\mathbf{x})|^{2}\rm{d}\mathbf{u}\rm{d}\mathbf{x}.
		\end{split}
	\end{align*}
Using the (19), the above formula becomes that
\begin{align*}
		\begin{split}
		&\int_{\mathbb{R}^t}\int_{\mathbb{R}^t}|\mathbf{w}|^{-\alpha}|G^{A_1,A_2}_{\phi}\{f\}(\mathbf{w,u})|^{2}\rm{d}\mathbf{u}\rm{d}\mathbf{w}\\
&\leq\frac{1}{4\pi^{2}|\mathbf{b}|^{\alpha}}M_{\alpha}\int_{\mathbb{R}^t}\int_{\mathbb{R}^t}|\mathbf{x}|^{\alpha}|f(\mathbf{x})
\overline{\phi(\mathbf{x-u})}|^{2}\rm{d}\mathbf{u}\rm{d}\mathbf{x}\\
&=\frac{1}{4\pi^{2}|\mathbf{b}|^{\alpha}}M_{\alpha}\int_{\mathbb{R}^t}|\mathbf{x}|^{\alpha}|f(\mathbf{x})|^{2}
\left(\int_{\mathbb{R}^t}|\overline{\phi(\mathbf{x-u})}|^{2}\rm{d}\mathbf{u}\right)\rm{d}\mathbf{x}.
		\end{split}
	\end{align*}
According to (6), then
\begin{align*}
		\begin{split}
		&\int_{\mathbb{R}^t}\int_{\mathbb{R}^t}|\mathbf{w}|^{-\alpha}|G^{A_1,A_2}_{\phi}\{f\}(\mathbf{w,u})|^{2}\rm{d}\mathbf{u}\rm{d}\mathbf{w}\\
&\leq\frac{1}{4\pi^{2}|\mathbf{b}|^{\alpha}}M_{\alpha}\|\phi\|^{2}_{L^2(\mathbb{R}^t,\mathbb{H})}
\int_{\mathbb{R}^t}|\mathbf{x}|^{\alpha}|\textit{f}(\mathbf{x})|^{2}\rm{d}\mathbf{x}.
		\end{split}
	\end{align*}
\end{proof}
\begin{corollary}
While $A_1=A_2=\begin{bmatrix}
	0&1\\
	-1&0
	\end{bmatrix}$,
the Pitt's inequality of the QWLCT leads to the Pitt's inequality of the two-sided quaternion
windowed Fourier transform (QWFT) \cite{35,40}, i.e.
\begin{align*}
		\begin{split}
		\int_{\mathbb{R}^t}\int_{\mathbb{R}^t}|\mathbf{w}|^{-\alpha}|G_{\phi}\{f\}(\mathbf{w,u})|^{2}\rm{d}\mathbf{u}\rm{d}\mathbf{w}
&\leq M_{\alpha}\|\phi\|^{2}_{L^2(\mathbb{R}^t,\mathbb{H})}
\int_{\mathbb{R}^t}|\mathbf{x}|^{\alpha}|\textit{f}(\mathbf{x})|^{2}\rm{d}\mathbf{x},
		\end{split}
	\end{align*}
where $G_{\phi}\{f\}(\mathbf{w,u})$ is the QWFT of the signal $f$.
\end{corollary}
\begin{corollary}
While $\phi=1$ and $A_1=A_2=\begin{bmatrix}
	0&1\\
	-1&0
	\end{bmatrix}$,
the Pitt's inequality of the QWLCT leads to the Pitt's inequality of the QFT.
\end{corollary}
\begin{remark}
When the window function is normalized, namely, $\|\phi\|_{L^{2}(\mathbb{R}^2,\mathbb{H})}=1$, then \eqref{pi} implies
\begin{align*}
		\begin{split}
		\int_{\mathbb{R}^t}\int_{\mathbb{R}^t}|\mathbf{w}|^{-\alpha}|G^{A_1,A_2}_{\phi}\{f\}(\mathbf{w,u})|^{2}\rm{d}\mathbf{u}\rm{d}\mathbf{w}
&\leq \frac{1}{4\pi^{2}|\mathbf{b}|^{\alpha}}M_{\alpha}
\int_{\mathbb{R}^t}|\mathbf{x}|^{\alpha}|\textit{f}(\mathbf{x})|^{2}\rm{d}\mathbf{x}.
		\end{split}
	\end{align*}
\end{remark}
\subsection{Lieb inequality}
\begin{theorem}
 Let $\phi \in L^2(\mathbb{R}^2,\mathbb{H})\backslash \{0\}$ be window functions. For every $f \in L^2\left( \mathbb{R}^2,\mathbb{H}\right)$ and $s\geq2$.
Then
\begin{align}\label{Li}
		\begin{split}
		\|G^{A_1,A_2}_{\phi}\{f\}\|_{L^s(\mathbb{R}^4,\mathbb{H})}
\leq \frac{|\mathbf{b}|^{\frac{1}{s}-\frac{1}{2}}}{2\pi}D_{s,s'}\|\textit{f}\|_{L^{2}(\mathbb{R}^2,\mathbb{H})}\|\phi\|_{L^{2}(\mathbb{R}^2,\mathbb{H})},
		\end{split}
	\end{align}
\end{theorem}
where $D_{s,s'}=\left( \frac{4}{s}\right)^{\frac{1}{s}}\left( \frac{4}{s'}\right)^{\frac{1}{s'}}$ and $\frac{1}{s}+\frac{1}{s'}=1$.
\begin{proof}
According to (18) and Hausdorff-Young theorem, we have
 \begin{align*}
		\begin{split}
		\left(\int_{\mathbb{R}^2}|G^{A_1,A_2}_{\phi}\{f\}(\mathbf{w,u})|^{s}\rm{d}\mathbf{w} \right)^{\frac{1}{s}}&=
\frac{1}{2\pi\sqrt{|\mathbf{b}|}}\left(\int_{\mathbb{R}^2}|F_{Q}(\textit{h})\left(\mathbf{\frac{w}{b},u}\right)|^{s}\rm{d}\mathbf{w} \right)^{\frac{1}{s}}\\
&\leq \frac{1}{2\pi\sqrt{|\mathbf{b}|}}|\mathbf{b}|^{\frac{1}{s}}\left(\int_{\mathbb{R}^2}|\textit{h}\left(\mathbf{x,u}\right)|^{s'}\rm{d}\mathbf{x} \right)^{\frac{1}{s'}}\\
&=\frac{|\mathbf{b}|^{\frac{1}{s}-\frac{1}{2}}}{2\pi}\left(\int_{\mathbb{R}^2}|\textit{f}\left(\mathbf{x}\right)|^{s'}|\phi\left(\mathbf{x-u}\right)|^{s'}\rm{d}\mathbf{x}
\right)^{\frac{1}{s'}}\\
&=\frac{|\mathbf{b}|^{\frac{1}{s}-\frac{1}{2}}}{2\pi}\left(|\textit{f}|^{s'}*|\widetilde{\phi}|^{s'}(\mathbf{x})\right)^{\frac{1}{s'}},
		\end{split}
	\end{align*}
where $\frac{1}{s}+\frac{1}{s'}=1$, $s\geq2$ and $\widetilde{\phi}(\mathbf{x})=\phi(-\mathbf{x})$.
 \begin{align*}
		\begin{split}
		\|G^{A_1,A_2}_{\phi}\{f\}\|_{L^s(\mathbb{R}^4,\mathbb{H})}&=\left(\int_{\mathbb{R}^2}\left(\int_{\mathbb{R}^2}|G^{A_1,A_2}_{\phi}\{f\}(\mathbf{w,u})|^{s}\rm{d}\mathbf{w} \right) \rm{d}\mathbf{u} \right)^{\frac{1}{s}}\\
&\leq \frac{|\mathbf{b}|^{\frac{1}{s}-\frac{1}{2}}}{2\pi}\left(\int_{\mathbb{R}^2}\left(|\textit{f}|^{s'}*|\widetilde{\phi}|^{s'}(\mathbf{x})\right)^{\frac{s}{s'}}\rm{d}\mathbf{u} \right)^{\frac{1}{s}}\\
&=\frac{|\mathbf{b}|^{\frac{1}{s}-\frac{1}{2}}}{2\pi}\||\textit{f}|^{s'}*|\widetilde{\phi}|^{s'}\|^{\frac{1}{s'}}_{L^{\frac{s}{s'}}(\mathbb{R}^2,\mathbb{H})}.
		\end{split}
	\end{align*}
If $k=\frac{2}{s'}$, $l=\frac{s}{s'}$ and $\frac{1}{k}+\frac{1}{k'}=1$, $\frac{1}{l}+\frac{1}{l'}=1$, then $\frac{1}{k}+\frac{1}{k}=1+\frac{1}{l}$ and hence as
$|\textit{f}|^{s'}$, $|\widetilde{\phi}|^{s'}\in L^k(\mathbb{R}^2,\mathbb{H})$ and by Young inequality, we have
\begin{align*}
		\begin{split}
		\||\textit{f}|^{s'}*|\widetilde{\phi}|^{s'}\|_{L^{l}(\mathbb{R}^2,\mathbb{H})}\leq B^{4}_{k}B^{2}_{l'}\||\textit{f}|^{s'}\|_{L^{k}(\mathbb{R}^2,\mathbb{H})}\||\widetilde{\phi}|^{s'}\|_{L^{k}(\mathbb{R}^2,\mathbb{H})},
		\end{split}
	\end{align*}
where $B_{s}=\left(\frac{s^{\frac{1}{s}}}{s'^{\frac{1}{s'}}} \right)^{\frac{1}{2}}$, $\frac{1}{s}+\frac{1}{s'}=1$.

On the other hand
\begin{align*}
		\begin{split}
		\||\textit{f}|^{s'}\|_{L^{k}(\mathbb{R}^2,\mathbb{H})}=\left(\int_{\mathbb{R}^2}|\textit{f}(\mathbf{x})|^{s'\frac{2}{s'}}\rm{d}\mathbf{x}\right)^{\frac{s'}{2}}=
\|\textit{f}\|^{s'}_{L^{2}(\mathbb{R}^2,\mathbb{H})},
		\end{split}
	\end{align*}
and
\begin{align*}
		\begin{split}
		\||\widetilde{\phi}|^{s'}\|_{L^{k}(\mathbb{R}^2,\mathbb{H})}=\left(\int_{\mathbb{R}^2}|\phi(\mathbf{x-u})|^{s'\frac{2}{s'}}\rm{d}\mathbf{x}\right)^{\frac{s'}{2}}=
\|\phi\|^{s'}_{L^{2}(\mathbb{R}^2,\mathbb{H})}.
		\end{split}
	\end{align*}
Hence, we obtain
 \begin{align*}
		\begin{split}
		\|G^{A_1,A_2}_{\phi}\{f\}\|_{L^s(\mathbb{R}^4,\mathbb{H})}
&\leq \frac{|\mathbf{b}|^{\frac{1}{s}-\frac{1}{2}}}{2\pi}\||\textit{f}|^{s'}*|\widetilde{\phi}|^{s'}\|^{\frac{1}{s'}}_{L^{\frac{s}{s'}}(\mathbb{R}^2,\mathbb{H})}\\
&\leq\frac{|\mathbf{b}|^{\frac{1}{s}-\frac{1}{2}}}{2\pi}\left(B^{4}_{k}B^{2}_{l'}\|\textit{f}\|^{s'}_{L^{2}(\mathbb{R}^2,\mathbb{H})}\|\phi\|^{s'}_{L^{2}(\mathbb{R}^2,\mathbb{H})} \right)^{\frac{1}{s'}}\\
&=\frac{|\mathbf{b}|^{\frac{1}{s}-\frac{1}{2}}}{2\pi}B^{\frac{4}{s'}}_{k}B^{\frac{2}{s'}}_{l'}\|\textit{f}\|_{L^{2}(\mathbb{R}^2,\mathbb{H})}\|\phi\|_{L^{2}(\mathbb{R}^2,\mathbb{H})}\\
&=\frac{|\mathbf{b}|^{\frac{1}{s}-\frac{1}{2}}}{2\pi}D_{s,s'}\|\textit{f}\|_{L^{2}(\mathbb{R}^2,\mathbb{H})}\|\phi\|_{L^{2}(\mathbb{R}^2,\mathbb{H})}.
		\end{split}
	\end{align*}
\end{proof}
\begin{corollary}
While $A_1=A_2=\begin{bmatrix}
	0&1\\
	-1&0
	\end{bmatrix}$,
the Lieb inequality of the QWLCT becomes the Lieb inequality of the QWFT \cite{41}.
\end{corollary}
\begin{remark}
When $\|\textit{f}\|_{L^{2}(\mathbb{R}^2,\mathbb{H})}=\|\phi\|_{L^{2}(\mathbb{R}^2,\mathbb{H})}=1$, then \eqref{Li} implies
\begin{align*}
		\begin{split}
		\|G^{A_1,A_2}_{\phi}\{f\}\|_{L^s(\mathbb{R}^4,\mathbb{H})}
\leq \frac{|\mathbf{b}|^{\frac{1}{s}-\frac{1}{2}}}{2\pi}D_{s,s'}.
		\end{split}
	\end{align*}
\end{remark}
\section{Uncertainty principles for the QWLCT}

\subsection{Logarithmic uncertainty principle}
Based on Pitt's inequality, Logarithmic uncertainty principle for the QFT has been proved in \cite{36}. In this subsection, applying Logarithmic uncertainty principle for the QFT, we study Logarithmic uncertainty principle for the QWLCT.
\begin{lemma}[Logarithmic uncertainty principle for the QFT] \cite{36} \label{LQFT}
For $f\in S(\mathbb{R}^2,\mathbb{H})$,
\begin{align}
		\begin{split}
\int_{\mathbb{R}^2}\ln|\mathbf{x}||f(\mathbf{x})|^{2}\rm{d}\mathbf{x}+\int_{\mathbb{R}^2}\ln|\mathbf{w}||\textit{F}_{Q}(\textit{f})(\mathbf{w})|^{2}\rm{d}\mathbf{w}\geq
\Delta\int_{\mathbb{R}^2}|\textit{f}(\mathbf{x})|^{2}\rm{d}\mathbf{x},
		\end{split}
	\end{align}
where $\Delta=\varphi(\frac{1}{2})-\ln(\pi)$, $\varphi(t)=\frac{\Gamma'(t)}{\Gamma(t)}$.
\end{lemma}
Now we arrive at the following result.
\begin{theorem}[Logarithmic uncertainty principle for the QWLCT]
 For $f, \phi\in S(\mathbb{R}^2,\mathbb{H})$,
\begin{align}
		\begin{split}
&\frac{\|\phi\|^{2}_{L^2(\mathbb{R}^2,\mathbb{H})}}{4\pi^{2}}\int_{\mathbb{R}^2}\ln|\mathbf{x}||f(\mathbf{x})|^{2}\rm{d}\mathbf{x}+
\int_{\mathbb{R}^2}\int_{\mathbb{R}^2}\ln|\mathbf{w}||G_{\phi}^{A_1,A_2}\{\textit{f}\}(\mathbf{w,u})|^{2}\rm{d}\mathbf{w}\rm{d}\mathbf{u}\\&\geq
\frac{(\Delta+\ln|\mathbf{b}|)}{4\pi^{2}}\|f\|_{L^2(\mathbb{R}^2,\mathbb{H})}^{2}\|\phi\|^{2}_{L^2(\mathbb{R}^2,\mathbb{H})},
		\end{split}
	\end{align}
where $\Delta=\varphi(\frac{1}{2})-\ln(\pi)$, $\varphi(t)=\frac{\Gamma'(t)}{\Gamma(t)}$.
\end{theorem}
\begin{proof}
According to (18), we have
\begin{align}\label{l}
		\begin{split}
&\int_{\mathbb{R}^2}\int_{\mathbb{R}^2}\ln|\mathbf{w}||G_{\phi}^{A_1,A_2}\{f\}(\mathbf{w,u})|^{2}\rm{d}\mathbf{w}\rm{d}\mathbf{u}\\&=
\frac{1}{4\pi^{2}|\mathbf{b}|}\int_{\mathbb{R}^2}\int_{\mathbb{R}^2}\ln|\mathbf{w}||F_{Q}(\textit{h})\left(\mathbf{\frac{w}{b},u}\right)|^{2}\rm{d}\mathbf{u}\rm{d}\mathbf{w}\\
&=\frac{1}{4\pi^{2}}\int_{\mathbb{R}^2}\int_{\mathbb{R}^2}\ln|\mathbf{b}\mathbf{y}||F_{Q}(\textit{h})\left(\mathbf{y,u}\right)|^{2}\rm{d}\mathbf{u}\rm{d}\mathbf{y}\\
&=\frac{1}{4\pi^{2}}\int_{\mathbb{R}^2}\int_{\mathbb{R}^2}\ln|\mathbf{b}||F_{Q}(\textit{h})\left(\mathbf{y,u}\right)|^{2}\rm{d}\mathbf{u}\rm{d}\mathbf{y}\\
&+\frac{1}{4\pi^{2}}\int_{\mathbb{R}^2}\int_{\mathbb{R}^2}\ln|\mathbf{y}||\textit{F}_{Q}(\textit{h})\left(\mathbf{y,u}\right)|^{2}\rm{d}\mathbf{u}\rm{d}\mathbf{y},
		\end{split}
	\end{align}
from Parseval's formula for QFT, the \eqref{l} becomes
\begin{align}\label{o}
		\begin{split}
&\int_{\mathbb{R}^2}\int_{\mathbb{R}^2}\ln|\mathbf{w}||G_{\phi}^{A_1,A_2}\{f\}(\mathbf{w,u})|^{2}\rm{d}\mathbf{w}\rm{d}\mathbf{u}
\\&=\frac{1}{4\pi^{2}}\ln|\mathbf{b}|\int_{\mathbb{R}^2}\int_{\mathbb{R}^2}|\textit{h}(\mathbf{w,u})|^{2}\rm{d}\mathbf{u}\rm{d}\mathbf{x}+
\frac{1}{4\pi^{2}}\int_{\mathbb{R}^2}\int_{\mathbb{R}^2}\ln|\mathbf{y}||\textit{F}_{Q}(\textit{h})\left(\mathbf{y,u}\right)|^{2}\rm{d}\mathbf{u}\rm{d}\mathbf{y}\\
&=\frac{\ln|\mathbf{b}|}{4\pi^{2}}\|f\|_{L^2(\mathbb{R}^2,\mathbb{H})}^{2}\|\phi\|^{2}_{L^2(\mathbb{R}^2,\mathbb{H})}+
\frac{1}{4\pi^{2}}\int_{\mathbb{R}^2}\int_{\mathbb{R}^2}\ln|\mathbf{y}||F_{Q}(\textit{h})\left(\mathbf{y,u}\right)|^{2}\rm{d}\mathbf{u}\rm{d}\mathbf{y}.
		\end{split}
	\end{align}
On the other hand, because $f, \phi\in S(\mathbb{R}^2,\mathbb{H})$, which implies $h(\mathbf{x,u})\in S(\mathbb{R}^2,\mathbb{H})$.
Therefore, according to the Lemma 4, we have
\begin{align}\label{t}
		\begin{split}
\int_{\mathbb{R}^2}\ln|\mathbf{x}||h(\mathbf{x})|^{2}\rm{d}\mathbf{x}+\int_{\mathbb{R}^2}\ln|\mathbf{w}||\textit{F}_{Q}(\textit{h})(\mathbf{w})|^{2}\rm{d}\mathbf{w}\geq
\Delta\int_{\mathbb{R}^2}|\textit{h}(\mathbf{x})|^{2}\rm{d}\mathbf{x},
		\end{split}
	\end{align}
let $h(\mathbf{x,u})=e^{\mathbf{i}\frac{a_1}{2b_1}x_1^2}f(\mathbf{x})\overline{\phi(\mathbf{x-u})}
e^{\mathbf{j}\frac{a_2}{2b_2}x_2^2}$, multiply both sides of the \eqref{t} by $\frac{1}{4\pi^{2}}$ and
integrating both sides of the \eqref{t} with respect to $\rm{d}\mathbf{u}$ , we have
\begin{align}\label{k}
		\begin{split}
&\frac{1}{4\pi^{2}}\int_{\mathbb{R}^2}\int_{\mathbb{R}^2}\ln|\mathbf{x}||h(\mathbf{x})|^{2}\rm{d}\mathbf{x}\rm{d}\mathbf{u}+
\frac{1}{4\pi^{2}}\int_{\mathbb{R}^2}\int_{\mathbb{R}^2}\ln|\mathbf{w}||\textit{F}_{Q}(\textit{h})(\mathbf{w})|^{2}\rm{d}\mathbf{w}\rm{d}\mathbf{u}\\&\geq
\frac{1}{4\pi^{2}}\Delta\int_{\mathbb{R}^2}\int_{\mathbb{R}^2}|\textit{h}(\mathbf{x})|^{2}\rm{d}\mathbf{x}\rm{d}\mathbf{u},
		\end{split}
	\end{align}
that is
\begin{align}\label{m}
		\begin{split}
&\frac{\|\phi\|^{2}_{L^2(\mathbb{R}^2,\mathbb{H})}}{4\pi^{2}}\int_{\mathbb{R}^2}\ln|\mathbf{x}||f(\mathbf{x})|^{2}\rm{d}\mathbf{x}+
\frac{1}{4\pi^{2}}\int_{\mathbb{R}^2}\int_{\mathbb{R}^2}\ln|\mathbf{w}||\textit{F}_{Q}(\textit{h})(\mathbf{w})|^{2}\rm{d}\mathbf{w}\rm{d}\mathbf{u}\\&\geq
\frac{\Delta}{4\pi^{2}}\|f\|_{L^2(\mathbb{R}^2,\mathbb{H})}^{2}\|\phi\|^{2}_{L^2(\mathbb{R}^2,\mathbb{H})},
		\end{split}
	\end{align}
Inserting \eqref{o} into the \eqref{m}, we obtain
\begin{align*}
		\begin{split}
&\frac{\|\phi\|^{2}_{L^2(\mathbb{R}^2,\mathbb{H})}}{4\pi^{2}}\int_{\mathbb{R}^2}\ln|\mathbf{x}||f(\mathbf{x})|^{2}\rm{d}\mathbf{x}+
\int_{\mathbb{R}^2}\int_{\mathbb{R}^2}\ln|\mathbf{w}||G_{\phi}^{A_1,A_2}\{\textit{f}\}(\mathbf{w,u})|^{2}\rm{d}\mathbf{w}\rm{d}\mathbf{u}\\&\geq
\frac{(\Delta+\ln|\mathbf{b}|)}{4\pi^{2}}\|f\|_{L^2(\mathbb{R}^2,\mathbb{H})}^{2}\|\phi\|^{2}_{L^2(\mathbb{R}^2,\mathbb{H})}.
		\end{split}
	\end{align*}
\end{proof}
\begin{corollary}
While $A_1=A_2=\begin{bmatrix}
	0&1\\
	-1&0
	\end{bmatrix}$,
the Logarithmic uncertainty principle of the QWLCT becomes the Logarithmic uncertainty principle of the QWFT \cite{40}.
\end{corollary}
\begin{corollary}
Suppose that $\phi=1$ and $A_1=A_2=\begin{bmatrix}
	0&1\\
	-1&0
	\end{bmatrix}$,
the Logarithmic uncertainty principle of the QWLCT becomes the Lemma \ref{LQFT}.
\end{corollary}
\subsection{Entropic uncertainty principle}
In this subsection, we present Entropic uncertainty principle for the QWLCT.
Entropic uncertainty principle is very important in signal processing and harmonic analysis. Its locality is measured by
Shannon entropy. The entropy represents an advantageous way to measure the decay of a function.
\begin{definition}
The Shannon entropy of $p$ is denoted as
\begin{align} \label{ed}
		\begin{split}
E(p)=-\int_{\mathbb{R}^2}\int_{\mathbb{R}^2}\ln(p(\mathbf{w,u}))p(\mathbf{w,u})\rm{d}\mathbf{w}\rm{d}\mathbf{u},
		\end{split}
	\end{align}
where $p$ is a probability density function on $\mathbb{R}^2\times \mathbb{R}^2$.
\end{definition}
\begin{theorem}
Let $\phi$ be a quaternion window function and $f\in L^2(\mathbb{R}^2,\mathbb{H})$ such that $f, \phi\neq0$ and $|b_{1}b_{2}|\geq\frac{1}{4\pi^{2}}$. Then
\begin{align} \label{EUP}
		\begin{split}
&E(|G_{\phi}^{A_1,A_2}\{f\}|^{2})\\
&\geq
\frac{1}{2\pi^{2}}(\ln(2)-\ln(2\pi)-\ln(|b_{1}b_{2}|))\|f\|^{2}_{L^2(\mathbb{R}^2,\mathbb{H})}\|\phi\|^{2}_{L^2(\mathbb{R}^2,\mathbb{H})}\\&
-\|f\|^{2}_{L^2(\mathbb{R}^2,\mathbb{H})}\|\phi\|^{2}_{L^2(\mathbb{R}^2,\mathbb{H})}\ln(\|f\|^{2}_{L^2(\mathbb{R}^2,\mathbb{H})}\|\phi\|^{2}_{L^2(\mathbb{R}^2,\mathbb{H})}).
		\end{split}
	\end{align}
\end{theorem}
\begin{proof}
Step 1: Suppose that $\|f\|_{L^2(\mathbb{R}^2,\mathbb{H})}=\|\phi\|_{L^2(\mathbb{R}^2,\mathbb{H})}=1$,
according to \eqref{wq}, we have
\begin{align} \label{bj}
		\begin{split}
&0\leq|G_{\phi}^{A_1,A_2}\{f\}(\mathbf{w,u})|\leq \|G_{\phi}^{A_1,A_2}\{f\}\|_{L^\infty(\mathbb{R}^4,\mathbb{H})}\leq \\&\frac{1}{2\pi\sqrt{|b_{1}b_{2}|}} \|f\|_{L^2(\mathbb{R}^2,\mathbb{H})}\|\phi\|_{L^2(\mathbb{R}^2,\mathbb{H})}\leq1,
		\end{split}
	\end{align}
then $\ln(|G_{\phi}^{A_1,A_2}\{f\}|)\leq0$ and $E(|G_{\phi}^{A_1,A_2}\{f\}|)\geq0$.

i. If the entropy $E(|G_{\phi}^{A_1,A_2}\{f\}|)=+\infty$, then the \eqref{EUP} hold.

ii. Now assume that the entropy $E(|G_{\phi}^{A_1,A_2}\{f\}|)<+\infty$.
Let $\Theta_{y}(z)=\frac{y^{z}-y^{2}}{z-2}$ be the function defined on $(2,3]$
and $0<y<1$, then for $\forall \ 2 < z\leq3$, we have
\begin{align} \label{ds}
		\begin{split}
\frac{\rm{d}\Theta_{y}}{\rm{d}z}(z)=\frac{(z-2)y^{z}\ln(y)-y^{z}+y^{2}}{(z-2)^{2}}.
		\end{split}
	\end{align}
Let $\hbar_{y}(z)=(z-2)y^{z}\ln(y)-y^{z}+y^{2}$, then for every $0<y<1$,
the function $\hbar_{y}$ is differentiable on $\mathbb{R}$, hence
\begin{align} \label{dds}
		\begin{split}
\frac{\rm{d}\hbar_{y}}{\rm{d}z}(z)=(z-2)(\ln(y))^{2}y^{z}.
		\end{split}
	\end{align}
We have that for all $0<y<1$, $\frac{\rm{d}\hbar_{y}}{\rm{d}z}(z)$ is positive on $(2,3]$,
so $\hbar_{y}$ is increasing on $(2,3]$.

In addition, for all $0<y<1$, $\lim_{z\rightarrow 2^{+}}\hbar_{y}(z)=\hbar_{y}(2)=0$, then $\hbar_{y}$ is positive which implies that $\frac{\rm{d}\Theta_{y}}{\rm{d}z}(z)$
is positive on $(2,3]$, hence $\Theta_{y}$ is increasing on $(2,3]$. Therefore $\forall \ 2 < z\leq3$, we have
\begin{align}
		\begin{split}
y^{2}\ln(y)=\lim_{z\rightarrow 2^{+}}\frac{y^{z}-y^{2}}{z-2}\leq \hbar_{y}(z),
		\end{split}
	\end{align}
that is to say $\forall \ 2 < z\leq3$,
\begin{align} \label{dsgs}
		\begin{split}
0\leq\frac{y^{2}-y^{z}}{z-2}\leq -y^{2}\ln(y).
		\end{split}
	\end{align}
When $y=0$ and $y=1$, the inequality \eqref{dsgs} holds. So for every $0\leq y\leq 1$, we have the inequality \eqref{dsgs}.

Step 2: According to \eqref{bj}, we know that for every $(\mathbf{w,u})\in \mathbb{R}^2\times \mathbb{R}^2$, $0\leq|G_{\phi}^{A_1,A_2}\{f\}(\mathbf{w,u})|\leq1$,
then for every $2 < z\leq3$, we obtain
\begin{align} \label{WDSGS}
		\begin{split}
0&\leq\frac{|G_{\phi}^{A_1,A_2}\{f\}(\mathbf{w,u})|^{2}-|G_{\phi}^{A_1,A_2}\{f\}(\mathbf{w,u})|^{z}}{z-2}\leq \\& -|G_{\phi}^{A_1,A_2}\{f\}(\mathbf{w,u})|^{2}\ln(|G_{\phi}^{A_1,A_2}\{f\}(\mathbf{w,u})|).
		\end{split}
	\end{align}
Let $\Phi$ be defined on $[2,+\infty)$ by
\begin{align} \label{xdhs}
		\begin{split}
\Phi(s)=\left( \int_{\mathbb{R}^2}\int_{\mathbb{R}^2}|G_{\phi}^{A_1,A_2}\{f\}(\mathbf{w,u})|^{s}\rm{d}\mathbf{w}\rm{d}\mathbf{u} \right)-
\left(\frac{|\mathbf{b}|^{\frac{1}{s}-\frac{1}{2}}}{2\pi}D_{s,s'}\right)^{s}.
		\end{split}
	\end{align}
From Lieb inequality \eqref{Li}, we know that $\Phi(s)\leq0$ for every $2\leq s<+\infty$ and using \eqref{fsxd}, we obtain
\begin{align}
		\begin{split}
\Phi(2)=1-\frac{1}{\pi^{2}}>0.
		\end{split}
	\end{align}
Hence $\left(\frac{\rm{d}\Phi}{\rm{d}s} \right)_{s=2^{+}}<0$.

If for every $2 < z\leq3$ and $(\mathbf{w,u})\in \mathbb{R}^2\times \mathbb{R}^2$, then
\begin{align*}
		\begin{split}
&\Big|\frac{|G_{\phi}^{A_1,A_2}\{f\}(\mathbf{w,u})|^{z}-|G_{\phi}^{A_1,A_2}\{f\}(\mathbf{w,u})|^{2}}{z-2}\Big|\leq \\& -|G_{\phi}^{A_1,A_2}\{f\}(\mathbf{w,u})|^{2}\ln(|G_{\phi}^{A_1,A_2}\{f\}(\mathbf{w,u})|).
		\end{split}
	\end{align*}
Then
\begin{align*}
		\begin{split}
&\int_{\mathbb{R}^2}\int_{\mathbb{R}^2}\Big|\frac{|G_{\phi}^{A_1,A_2}\{f\}(\mathbf{w,u})|^{z}-|G_{\phi}^{A_1,A_2}\{f\}(\mathbf{w,u})|^{2}}{z-2}\Big|\rm{d}\mathbf{w}\rm{d}\mathbf{u}
\leq \\& -\int_{\mathbb{R}^2}\int_{\mathbb{R}^2}|G_{\phi}^{A_1,A_2}\{f\}(\mathbf{w,u})|^{2}\ln(|G_{\phi}^{A_1,A_2}\{f\}(\mathbf{w,u})|)\rm{d}\mathbf{w}\rm{d}\mathbf{u}\\
&=\frac{1}{2}E(|G_{\phi}^{A_1,A_2}\{f\}(\mathbf{w,u})|^{2})<+\infty.
		\end{split}
	\end{align*}
If for every $3 \leq z<+\infty$ and $(\mathbf{w,u})\in \mathbb{R}^2\times \mathbb{R}^2$, then
\begin{align*}
		\begin{split}
&\Big|\frac{|G_{\phi}^{A_1,A_2}\{f\}(\mathbf{w,u})|^{z}-|G_{\phi}^{A_1,A_2}\{f\}(\mathbf{w,u})|^{2}}{z-2}\Big|\leq \\& 2|G_{\phi}^{A_1,A_2}\{f\}(\mathbf{w,u})|^{2},
		\end{split}
	\end{align*}
and
\begin{align*}
		\begin{split}
&\int_{\mathbb{R}^2}\int_{\mathbb{R}^2}\Big|\frac{|G_{\phi}^{A_1,A_2}\{f\}(\mathbf{w,u})|^{z}-|G_{\phi}^{A_1,A_2}\{f\}(\mathbf{w,u})|^{2}}{z-2}\Big|\rm{d}\mathbf{w}\rm{d}\mathbf{u}
\\&\leq 2\|G_{\phi}^{A_1,A_2}\{\textit{f}\}\|_{L^2(\mathbb{R}^4,\mathbb{H})}<+\infty.
		\end{split}
	\end{align*}
Therefore for every $2 \leq z<+\infty$ and $(\mathbf{w,u})\in \mathbb{R}^2\times \mathbb{R}^2$, we have
\begin{align}
		\begin{split}
\int_{\mathbb{R}^2}\int_{\mathbb{R}^2}\Big|\frac{|G_{\phi}^{A_1,A_2}\{f\}(\mathbf{w,u})|^{z}-|G_{\phi}^{A_1,A_2}\{f\}(\mathbf{w,u})|^{2}}{z-2}\Big|\rm{d}\mathbf{w}\rm{d}\mathbf{u}
<+\infty.
		\end{split}
	\end{align}
According to \eqref{WDSGS}, We obtain
\begin{align*}
		\begin{split}
&\Big(\frac{\rm{d}}{\rm{d}z}\int_{\mathbb{R}^2}\int_{\mathbb{R}^2}|G_{\phi}^{A_1,A_2}\{f\}(\mathbf{w,u})|^{z}\rm{d}\mathbf{w}\rm{d}\mathbf{u}\Big)_{z=2^{+}}\\
&=\lim_{z \rightarrow 2^{+}}\int_{\mathbb{R}^2}\int_{\mathbb{R}^2}\frac{|G_{\phi}^{A_1,A_2}\{f\}(\mathbf{w,u})|^{z}-|G_{\phi}^{A_1,A_2}\{f\}(\mathbf{w,u})|^{2}}{z-2}
\rm{d}\mathbf{w}\rm{d}\mathbf{u}\\&
= \int_{\mathbb{R}^2}\int_{\mathbb{R}^2}\lim_{z \rightarrow 2^{+}}\frac{|G_{\phi}^{A_1,A_2}\{f\}(\mathbf{w,u})|^{z}-|G_{\phi}^{A_1,A_2}\{f\}(\mathbf{w,u})|^{2}}{z-2}
\rm{d}\mathbf{w}\rm{d}\mathbf{u}\\&=
\frac{1}{2}\int_{\mathbb{R}^2}\int_{\mathbb{R}^2}|G_{\phi}^{A_1,A_2}\{f\}(\mathbf{w,u})|^{2}\ln(|G_{\phi}^{A_1,A_2}\{f\}(\mathbf{w,u})|^{2})\rm{d}\mathbf{w}\rm{d}\mathbf{u}\\
&=-\frac{1}{2}E(|G_{\phi}^{A_1,A_2}\{f\}(\mathbf{w,u})|^{2}).
		\end{split}
	\end{align*}
Moreover, by \eqref{xdhs}, we have
\begin{align*}
		\begin{split}
\Big(\frac{\rm{d}\Phi}{\rm{d}z}\Big)_{z=2^{+}}=-\frac{1}{2}E(|G_{\phi}^{A_1,A_2}\{f\}(\mathbf{w,u})|^{2})-
\left(\frac{\rm{d}}{\rm{d}z}\left(\frac{|\mathbf{b}|^{\frac{1}{z}-\frac{1}{2}}}{2\pi}D_{z,z'}\right)^{z}\right)_{z=2^{+}},
		\end{split}
	\end{align*}
where $D_{z,z'}=\left( \frac{4}{z}\right)^{\frac{1}{z}}\left( \frac{4}{z'}\right)^{\frac{1}{z'}}$ and $\frac{1}{z}+\frac{1}{z'}=1$.

On the other hand,
\begin{align*}
		\begin{split}
\left(\frac{\rm{d}}{\rm{d}z}\left(\frac{|\mathbf{b}|^{\frac{1}{z}-\frac{1}{2}}}{2\pi}D_{z,z'}\right)^{z}\right)_{z=2^{+}}=\frac{1}{4\pi^{2}}(\ln(2\pi)+\ln(|b_{1}b_{2}|)-\ln(2)).
		\end{split}
	\end{align*}
Hence
\begin{align*}
		\begin{split}
&\Big(\frac{\rm{d}\Phi}{\rm{d}z}\Big)_{z=2^{+}}\\&=-\frac{1}{2}E(|G_{\phi}^{A_1,A_2}\{f\}(\mathbf{w,u})|^{2})+
\frac{1}{4\pi^{2}}(\ln(2)-\ln(2\pi)-\ln(|b_{1}b_{2}|))\leq0,
		\end{split}
	\end{align*}
we obtain
\begin{align*}
		\begin{split}
E(|G_{\phi}^{A_1,A_2}\{f\}(\mathbf{w,u})|^{2})\geq
\frac{1}{2\pi^{2}}(\ln(2)-\ln(2\pi)-\ln(|b_{1}b_{2}|)).
		\end{split}
	\end{align*}
Step 3: For generic $f, \phi\neq0$, Let $\zeta=\frac{f}{\|f\|_{L^2(\mathbb{R}^2,\mathbb{H})}}$ and $\vartheta=\frac{\phi}{\|\phi\|_{L^2(\mathbb{R}^2,\mathbb{H})}}$, such that
$\|\zeta\|_{L^2(\mathbb{R}^2,\mathbb{H})}=\|\vartheta\|_{L^2(\mathbb{R}^2,\mathbb{H})}=1$ and $E(|G_{\vartheta}^{A_1,A_2}\{\zeta\}(\mathbf{w,u})|^{2})\geq
\frac{1}{2\pi^{2}}(\ln(2)-\ln(2\pi)-\ln(|b_{1}b_{2}|))$.

According to the definition of the QWLCT, we know that
\begin{align*}
		\begin{split}
G_{\vartheta}^{A_1,A_2}\{\zeta\}(\mathbf{w,u})=\frac{G_{\phi}^{A_1,A_2}\{f\}(\mathbf{w,u})}{\|f\|_{L^2(\mathbb{R}^2,\mathbb{H})}\|\phi\|_{L^2(\mathbb{R}^2,\mathbb{H})}}.
		\end{split}
	\end{align*}
Hence
\begin{align*}
		\begin{split}
E(|G_{\vartheta}^{A_1,A_2}\{\zeta\}|^{2})&=-\int_{\mathbb{R}^2}\int_{\mathbb{R}^2}(\ln(|G_{\phi}^{A_1,A_2}\{f\}(\mathbf{w,u})|^{2})\\&-
\ln(\|f\|^{2}_{L^2(\mathbb{R}^2,\mathbb{H})}\|\phi\|^{2}_{L^2(\mathbb{R}^2,\mathbb{H})}))\frac{|G_{\phi}^{A_1,A_2}\{f\}(\mathbf{w,u})|^{2}}
{\|f\|^{2}_{L^2(\mathbb{R}^2,\mathbb{H})}\|\phi\|^{2}_{L^2(\mathbb{R}^2,\mathbb{H})})}\rm{d}\mathbf{w}\rm{d}\mathbf{u}\\
&=\frac{E(|G_{\phi}^{A_1,A_2}\{f\}|^{2})}{\|f\|^{2}_{L^2(\mathbb{R}^2,\mathbb{H})}\|\phi\|^{2}_{L^2(\mathbb{R}^2,\mathbb{H})})}
+\ln(\|f\|^{2}_{L^2(\mathbb{R}^2,\mathbb{H})}\|\phi\|^{2}_{L^2(\mathbb{R}^2,\mathbb{H})})\\
&\geq
\frac{1}{2\pi^{2}}(\ln(2)-\ln(2\pi)-\ln(|b_{1}b_{2}|)),
		\end{split}
	\end{align*}
consequently
\begin{align*}
		\begin{split}
&E(|G_{\phi}^{A_1,A_2}\{f\}|^{2})\\
&\geq
\frac{1}{2\pi^{2}}(\ln(2)-\ln(2\pi)-\ln(|b_{1}b_{2}|))\|f\|^{2}_{L^2(\mathbb{R}^2,\mathbb{H})}\|\phi\|^{2}_{L^2(\mathbb{R}^2,\mathbb{H})}\\&
-\|f\|^{2}_{L^2(\mathbb{R}^2,\mathbb{H})}\|\phi\|^{2}_{L^2(\mathbb{R}^2,\mathbb{H})}\ln(\|f\|^{2}_{L^2(\mathbb{R}^2,\mathbb{H})}\|\phi\|^{2}_{L^2(\mathbb{R}^2,\mathbb{H})}).
		\end{split}
	\end{align*}
\end{proof}
\begin{corollary}
While $A_1=A_2=\begin{bmatrix}
	0&1\\
	-1&0
	\end{bmatrix}$,
the Entropic uncertainty principle of the QWLCT becomes the  Entropic uncertainty principle of the QWFT \cite{41}.
\end{corollary}
\subsection{Lieb uncertainty principle}
In this subsection, we prove Lieb uncertainty principle of the QWLCT based on the following definition and Lieb inequality.
\begin{definition}
A function $f\in L^2(\mathbb{R}^2,\mathbb{H})$ be $\epsilon-$concentrated on a measurable set $\Omega\subseteq\mathbb{R}^2$,
where $\Omega^{c}=\mathbb{R}^2\ \Omega$, if
\begin{align}
		\begin{split}
\|\chi_{\Omega^{c}}\textit{f}\|_{L^2(\mathbb{R}^2,\mathbb{H})}=\left(\int_{\Omega^{c}}|\textit{f}(\mathbf{x})|^{2}\rm{d}\mathbf{x}\right)^{\frac{1}{2}}\leq \varepsilon
\|f\|_{L^2(\mathbb{R}^2,\mathbb{H})},
		\end{split}
	\end{align}
where $\chi_{\Omega}$ denote the characteristic function of $\Omega$. If $0\leq\epsilon\leq\frac{1}{2}$, then the most of energy is concentrated on $\Omega$ and $\Omega$ can be called the
essential support of $f$. If $\epsilon=0$, then $\Omega$ contains the support of $f$.
\end{definition}

\begin{theorem} [Lieb uncertainty principle for the QWLCT]
Let $\phi$ be a quaternion window function and $f\in L^2(\mathbb{R}^2,\mathbb{H})$ such that $f\neq0$.
Let $\Omega$ be a measurable set of $\mathbb{R}^2\times \mathbb{R}^2$ and $\epsilon\geq0$.
If $G_{\phi}^{A_1,A_2}(f)$ is $\epsilon-$concentrated on $\Omega$, then for every $s>2$, we have
\begin{align}
		\begin{split}
|\mathbf{b}|(1-\varepsilon^{2})^{\frac{s}{s-2}}\left(\frac{D_{s,s'}}{2\pi}\right)^{\frac{2s}{2-s}}
&\leq |\Omega|,
		\end{split}
	\end{align}
where $|\Omega|$ is the normalized Lebesgue measure of the set $\Omega$, $D_{s,s'}=\left( \frac{4}{s}\right)^{\frac{1}{s}}\left( \frac{4}{s'}\right)^{\frac{1}{s'}}$ and $\frac{1}{s}+\frac{1}{s'}=1$.
\end{theorem}
\begin{proof}
Because $G_{\phi}^{A_1,A_2}\{f\}$ is $\epsilon-$concentrated on $\Omega$, according to the Definition 4 and the \eqref{fsxd}, it implies that
\begin{align}
		\begin{split}
\|\chi_{\Omega^{c}}G_{\phi}^{A_1,A_2}\{f\}\|_{L^2(\mathbb{R}^4,\mathbb{H})}
\leq \varepsilon \|G_{\phi}^{A_1,A_2}\{f\}\|_{L^2(\mathbb{R}^4,\mathbb{H})}=\varepsilon\|f\|_{L^2(\mathbb{R}^2,\mathbb{H})}\|\phi\|_{L^2(\mathbb{R}^2,\mathbb{H})},
		\end{split}
	\end{align}
hence
\begin{align}\label{dd}
		\begin{split}
\|\chi_{\Omega}G_{\phi}^{A_1,A_2}\{f\}\|^{2}_{L^2(\mathbb{R}^4,\mathbb{H})}&=\|G_{\phi}^{A_1,A_2}\{f\}\|^{2}_{L^2(\mathbb{R}^4,\mathbb{H})}-
\|\chi_{\Omega^{c}}G_{\phi}^{A_1,A_2}\{f\}\|^{2}_{L^2(\mathbb{R}^4,\mathbb{H})}\\&
\geq (1-\varepsilon^{2})\|f\|^{2}_{L^2(\mathbb{R}^2,\mathbb{H})}\|\phi\|^{2}_{L^2(\mathbb{R}^2,\mathbb{H})},
		\end{split}
	\end{align}
and applying H$\ddot{o}$lder inequality, then
\begin{align}
		\begin{split}
\|\chi_{\Omega}G_{\phi}^{A_1,A_2}\{f\}\|^{2}_{L^2(\mathbb{R}^4,\mathbb{H})}
&=\int_{\mathbb{R}^2}\int_{\mathbb{R}^2}\chi_{\Omega}(\mathbf{w,u})|G_{\phi}^{A_1,A_2}\{f\}(\mathbf{w,u})|^{2}\rm{d}\mathbf{w}\rm{d}\mathbf{u}\\
&\leq \Big(\int_{\mathbb{R}^2}\int_{\mathbb{R}^2}(\chi_{\Omega}(\mathbf{w,u}))^{\frac{s}{s-2}}\rm{d}\mathbf{w}\rm{d}\mathbf{u}\Big)^{\frac{s-2}{s}}\\
&\times \Big(\int_{\mathbb{R}^2}\int_{\mathbb{R}^2}(|G_{\phi}^{A_1,A_2}\{f\}(\mathbf{w,u})|^{2})^{\frac{s}{2}}\rm{d}\mathbf{w}\rm{d}\mathbf{u}\Big)^{\frac{2}{s}}\\
&=|\Omega|^{\frac{s-2}{s}} \times \|G_{\phi}^{A_1,A_2}\{f\}\|^{2}_{L^s(\mathbb{R}^4,\mathbb{H})}.
		\end{split}
	\end{align}
In addition, by Lieb inequality, the above inequality deduce that
\begin{align}\label{w}
		\begin{split}
&\|\chi_{\Omega}G_{\phi}^{A_1,A_2}\{f\}\|^{2}_{L^2(\mathbb{R}^4,\mathbb{H})}\\
&\leq |\Omega|^{\frac{s-2}{s}}
\Big(\frac{|\mathbf{b}|^{\frac{1}{s}-\frac{1}{2}}}{2\pi}D_{s,s'}\Big)^{2}\|\textit{f}\|^{2}_{L^{2}(\mathbb{R}^2,\mathbb{H})}\|\phi\|^{2}_{L^{2}(\mathbb{R}^2,\mathbb{H})}.
		\end{split}
	\end{align}
According to \eqref{dd} and \eqref{w}, we obtain
\begin{align}
		\begin{split}
1-\varepsilon^{2}
&\leq |\Omega|^{\frac{s-2}{s}}
\Big(\frac{|\mathbf{b}|^{\frac{1}{s}-\frac{1}{2}}}{2\pi}D_{s,s'}\Big)^{2}.
		\end{split}
	\end{align}
So
\begin{align}
		\begin{split}
|\mathbf{b}|(1-\varepsilon^{2})^{\frac{s}{s-2}}\left(\frac{D_{s,s'}}{2\pi}\right)^{\frac{2s}{2-s}}
&\leq |\Omega|.
		\end{split}
	\end{align}
\end{proof}
\begin{corollary}
While $A_1=A_2=\begin{bmatrix}
	0&1\\
	-1&0
	\end{bmatrix}$,
the  Lieb uncertainty principle of the QWLCT results in the Lieb uncertainty principle of the QWFT \cite{41}.
\end{corollary}
\subsection{Donoho-Stark's uncertainty principle}
In this subsection, according to the relationship between the QWLCT and the QFT, we obtain Donoho-Stark's uncertainty principle for QWLCT.
\begin{lemma}[Donoho-Stark's uncertainty principle for QFT]\cite{36} \label{dqft}
Let $f\in L^2(\mathbb{R}^2,\mathbb{H})$ with $f\neq0$ is $\epsilon_{\Lambda}-$concentrated on $\Lambda \subseteq\mathbb{R}^2$,
and $F_{Q}(f)$ is $\epsilon_{\Xi}-$concentrated on $\Xi \subseteq\mathbb{R}^2$. Then
\begin{align}
		\begin{split}
|\Lambda||\Xi|\geq2\pi (1-\epsilon_{\Lambda}-\epsilon_{\Xi})^{2}.
		\end{split}
	\end{align}
\end{lemma}
Inspired by the Donoho-Stark's uncertainty principle for QFT, we can obtain the Donoho-Stark's uncertainty principle for QWLCT.
\begin{theorem}[Donoho-Stark's uncertainty principle for QWLCT]
Suppose that the nonzero signal quaternion function $\mathcal{L}_{A_1,A_2}^{-1}[G_{\phi}^{A_1,A_2}\{f\}](\mathbf{x,u})\in L^2(\mathbb{R}^2,\mathbb{H})$
is $\epsilon_{\Lambda}-$concentrated on $\Lambda \subseteq\mathbb{R}^2$, and $G_{\phi}^{A_1,A_2}\{f\}(\mathbf{w,u})$ is \\
$\epsilon_{\Xi}-$concentrated on $\Xi \subseteq\mathbb{R}^2$.
Then
\begin{align}
		\begin{split}
|\Lambda||\mathbf{b}\Xi|\geq2\pi (1-\epsilon_{\Lambda}-\epsilon_{\Xi})^{2}.
		\end{split}
	\end{align}
\end{theorem}
\begin{proof}
By applying \eqref{hn} and \eqref{hf}, we have
\begin{align}
		\begin{split}
		h(\mathbf{x,u})&=e^{\mathbf{i}\frac{a_1}{2b_1}x_1^2}\mathcal{L}_{A_1,A_2}^{-1}\{G^{A_1,A_2}_{\phi}\{f\}\}(\mathbf{x,u})e^{\mathbf{j}\frac{a_2}{2b_2}x_2^2}.
		\end{split}
	\end{align}
Since $\mathcal{L}_{A_1,A_2}^{-1}[G_{\phi}^{A_1,A_2}\{f\}](\mathbf{x,u})\in L^2(\mathbb{R}^2,\mathbb{H})$
is $\epsilon_{\Lambda}-$concentrated on $\Lambda \subseteq\mathbb{R}^2$, then
$h(\mathbf{x,u})\in L^2(\mathbb{R}^2,\mathbb{H})$
is $\epsilon_{\Lambda}-$concentrated on $\Lambda \subseteq\mathbb{R}^2$.

On the other hand, from $G_{\phi}^{A_1,A_2}\{f\}(\mathbf{w,u})$ is $\epsilon_{\Xi}-$concentrated on $\Xi \subseteq\mathbb{R}^2$,
based on the relationship (18), we know that $F_{Q}(h)\left(\frac{\mathbf{w}}{\mathbf{b}},\mathbf{u}\right)$ is $\epsilon_{\Xi}-$concentrated on $\Xi \subseteq\mathbb{R}^2$,
that is to say, $F_{Q}(h)\left(\mathbf{w},\mathbf{u}\right)$ is $\epsilon_{\Xi}-$concentrated on $\mathbf{b}\Xi \subseteq\mathbb{R}^2$.

Hence, according to Lemma 5, we obtain
\begin{align*}
		\begin{split}
|\Lambda||\mathbf{b}\Xi|\geq2\pi (1-\epsilon_{\Lambda}-\epsilon_{\Xi})^{2}.
		\end{split}
	\end{align*}
\end{proof}
\begin{corollary}
While $\mathcal{L}_{A_1,A_2}^{-1}[G_{\phi}^{A_1,A_2}\{f\}](\mathbf{x,u})\in L^2(\mathbb{R}^2,\mathbb{H})$,\\ $supp\mathcal{L}_{A_1,A_2}^{-1}[G_{\phi}^{A_1,A_2}\{f\}](\mathbf{x,u})\subseteq\Lambda$ and
$suppG_{\phi}^{A_1,A_2}\{f\}(\mathbf{w,u})\subseteq\Xi$, then
\begin{align*}
		\begin{split}
|\Lambda||\mathbf{b}\Xi|\geq2\pi .
		\end{split}
	\end{align*}
\end{corollary}
\section{ Numerical example and potential application}
The uncertainty principles for the QWLCT have been
studied in this paper. In this section, to show the correctness and useful of the theorems, an numerical
example is given to verify the result, and the potential applications are also presented
to show the importance of the theorems.

\textbf{Example:} Consider a signal
\begin{align*}
		\begin{split}
f(\mathbf{x})=\frac{1}{(\pi\beta)^{\frac{1}{2}}}e^{ix_{1}u_{0}}e^{-\frac{|\mathbf{x}|^{2}}{2\beta}}e^{ix_{2}v_{0}},
		\end{split}
	\end{align*}
and the window function is
\begin{align*}
		\begin{split}
\phi(\mathbf{x})=2\sqrt{\frac{\pi}{\beta}}e^{ix_{1}u_{0}}e^{-\frac{|\mathbf{x}|^{2}}{2\beta}}e^{ix_{2}v_{0}},
		\end{split}
	\end{align*}
where $\beta>0$,$u_{0}$ and $v_{0}$ are real constants.

It follows that
\begin{align*}
		\begin{split}
\|f\|_{L^2(\mathbb{R}^2,\mathbb{H})}^{2}&=\int_{\mathbb{R}^2}|f(\mathbf{x})|^{2}\rm{d}\mathbf{x}\\
&=\frac{1}{\pi\beta}\int_{\mathbb{R}^2}e^{-\frac{|\mathbf{x}|^{2}}{\beta}}\rm{d}\mathbf{x}\\
&=1,
		\end{split}
	\end{align*}
\begin{align*}
		\begin{split}
\|\phi\|^{2}_{L^2(\mathbb{R}^2,\mathbb{H})}&=\int_{\mathbb{R}^2}|\phi(\mathbf{x})|^{2}\rm{d}\mathbf{x}\\
&=\frac{4\pi^{2}}{\pi\beta}\int_{\mathbb{R}^2}e^{-\frac{|\mathbf{x}|^{2}}{\beta}}\rm{d}\mathbf{x}\\
&=4\pi^{2}.
		\end{split}
	\end{align*}
Therefore, the Logarithmic uncertainty principle for the QWLCT becomes
\begin{align}
		\begin{split}
&\int_{\mathbb{R}^2}\ln|\mathbf{x}||f(\mathbf{x})|^{2}\rm{d}\mathbf{x}+
\int_{\mathbb{R}^2}\int_{\mathbb{R}^2}\ln|\mathbf{w}||G_{\phi}^{A_1,A_2}\{\textit{f}\}(\mathbf{w,u})|^{2}\rm{d}\mathbf{w}\rm{d}\mathbf{u}\\&\geq
\Delta+\ln|\mathbf{b}|.
		\end{split}
	\end{align}
By Jensen's inequality \cite{36}, Logarithmic uncertainty principle for the QWLCT implies
Heisenberg-Weyl's uncertainty principle for the QWLCT
\begin{align}
		\begin{split}
\left(\int_{\mathbb{R}^2}|\mathbf{x}|^{2}|f(\mathbf{x})|^{2}\rm{d}\mathbf{x}
\int_{\mathbb{R}^2}\int_{\mathbb{R}^2}|\mathbf{w}|^{2}
|G_{\phi}^{A_1,A_2}\{\textit{f}\}(\mathbf{w,u})|^{2}\rm{d}\mathbf{w}\rm{d}\mathbf{u}\right)^{\frac{1}{2}}\geq
\frac{|\textbf{b}|}{4\pi}.
		\end{split}
	\end{align}
Let $A_1=A_2=\begin{bmatrix}
	0&\frac{1}{4}\\
	-\frac{1}{4}&1
	\end{bmatrix}$, $u_{0}=v_{0}=0$ and $\beta=\frac{1}{16}$, then
\begin{align*}
		\begin{split}
		h(\mathbf{x,u})&=e^{\mathbf{i}\frac{a_1}{2b_1}x_1^2}f(\mathbf{x})\overline{\phi(\mathbf{x-u})}
e^{\mathbf{j}\frac{a_2}{2b_2}x_2^2}\\
&=\frac{2}{\beta}
e^{-\frac{|\mathbf{x}|^{2}}{2\beta}}e^{-\frac{|\mathbf{x-u}|^{2}}{2\beta}},
		\end{split}
	\end{align*}
so
\begin{align*}
		\begin{split}
		F_{Q}(h)\left(\frac{\mathbf{w}}{\mathbf{b}},\mathbf{u}\right)&=
\int_{\mathbb{R}^2}e^{-\mathbf{i}x_1\frac{\omega_1}{b_{1}}}\frac{2}{\beta}
e^{-\frac{|\mathbf{x}|^{2}}{2\beta}}e^{-\frac{|\mathbf{x-u}|^{2}}{2\beta}}
e^{-\mathbf{j}x_2\frac{\omega_1}{b_{1}}}\rm{d}\mathbf{x}\\
&=2\pi e^{-\frac{u_{1}^{2}}{4\beta}-\frac{\beta w^{2}_{1}}{4b^{2}_{1}}-\textbf{i}\frac{u_{1}w_{1}}{2b_{1}}}
e^{-\frac{u_{1}^{2}}{4\beta}-\frac{\beta w^{2}_{2}}{4b^{2}_{2}}-\textbf{j}\frac{u_{2}w_{2}}{2b_{2}}}
		\end{split}
	\end{align*}
According to (18), we have
\begin{align*}
		\begin{split}
		&G^{A_{1},A_{2}}_{\phi}\{f\}(\mathbf{w,u})\\&=\frac{1}{\sqrt{ b_{1}b_2}}e^{\mathbf{i}\left(\frac{d_1}{2b_1}\omega_1^2-\frac{\pi}{4}-\frac{u_{1}w_{1}}{2b_{1}} \right)}e^{-\frac{u_{1}^{2}}{4\beta}-\frac{\beta w^{2}_{1}}{4b^{2}_{1}}}
e^{-\frac{u_{1}^{2}}{4\beta}-\frac{\beta w^{2}_{2}}{4b^{2}_{2}}}
e^{\mathbf{j}\left(\frac{d_2}{2}\omega_2^2-\frac{\pi}{4}-\frac{u_{2}w_{2}}{2b_{2}} \right)}.
		\end{split}
	\end{align*}
Moreover
\begin{align}
		\begin{split}
\int_{\mathbb{R}^2}|\mathbf{x}|^{2}|f(\mathbf{x})|^{2}\rm{d}\mathbf{x}=
\frac{1}{\pi\beta}\int_{\mathbb{R}^2}(x^{2}_{1}+x^{2}_{1})e^{-\frac{|\mathbf{x}|^{2}}{\beta}}\rm{d}\mathbf{x}
=\beta=\frac{1}{16},
		\end{split}
	\end{align}
and
\begin{align}
		\begin{split}
&\int_{\mathbb{R}^2}\int_{\mathbb{R}^2}|\mathbf{w}|^{2}
|G_{\phi}^{A_1,A_2}\{\textit{f}\}(\mathbf{w,u})|^{2}\rm{d}\mathbf{w}\rm{d}\mathbf{u}
\\&=\frac{1}{b_{1}b_{2}}\int_{\mathbb{R}^2}\int_{\mathbb{R}^2}(w^{2}_{1}+w^{2}_{1})
e^{-\frac{u_{1}^{2}}{2\beta}-\frac{\beta w^{2}_{1}}{2b^{2}_{1}}}
e^{-\frac{u_{1}^{2}}{2\beta}-\frac{\beta w^{2}_{2}}{2b^{2}_{2}}}\rm{d}\mathbf{w}\rm{d}\mathbf{u}\\
&=\frac{4\pi^{2}}{b_{1}b_{2}}(b^{3}_{1}\sqrt{b_{2}}+b^{3}_{2}\sqrt{b_{1}})\\
&=\frac{\pi^{2}}{16^{2}}.
		\end{split}
	\end{align}
Hence
\begin{align}
		\begin{split}
\left(\int_{\mathbb{R}^2}|\mathbf{x}|^{2}|f(\mathbf{x})|^{2}\rm{d}\mathbf{x}
\int_{\mathbb{R}^2}\int_{\mathbb{R}^2}|\mathbf{w}|^{2}
|G_{\phi}^{A_1,A_2}\{\textit{f}\}(\mathbf{w,u})|^{2}\rm{d}\mathbf{w}\rm{d}\mathbf{u}\right)^{\frac{1}{2}}=\frac{\pi}{4}.
		\end{split}
	\end{align}
On the other hand
\begin{align}
		\begin{split}
\frac{|\textbf{b}|}{4\pi}=\frac{\sqrt{2}}{16\pi}.
		\end{split}
	\end{align}
Then, expressions (67) and (68) verify (64). This means that Theorem 3 is verified.

The uncertainty principle has some applications in signal recovery. In \cite{42,49}, the authors studied the problem of signal recovery by uncertainty principles.
The paper \cite{45} evaluated in signal recovery problems where there is an interplay of missing and time-limiting data.
The correlative result has been applied to the window Fourier transform (WFT) as well in \cite{44}.
Recently, the authors \cite{46} extended signal recovery by using local uncertainty principle to the QLCT \cite{49}.
In this section, we will give a potential application for signal recovery by using Donoho-Stark's uncertainty principle for QWLCT.
Let the modified signal $f_{\mathbf{u}}(\mathbf{x})\in L^2(\mathbb{R}^2,\mathbb{H})$ is bandlimited to a set $Q$ of finite measure for the QLCT.
Assume that the receiver is unable to observe all of $f_{\mathbf{u}}(\mathbf{x})$, a certain subset $T$
of $x$-values is unobserved.
Furthermore, the modified signal $f_{\mathbf{u}}(\mathbf{x})$ is transmitted and the received signal is corrupted by observational noise
$n(\mathbf{x})$. Thus the received signal $r(\mathbf{x})$ can be taken to have the form
\begin{align}
		\begin{split}
			r(\mathbf{x})=\begin{cases}
		f_{\mathbf{u}}(\mathbf{x})+n(\mathbf{x}),   &\mathbf{x}\notin T  \\
		0,    &\mathbf{x}\in T
		\end{cases}
		\end{split}
	\end{align}
The receiver's aim is to reconstruct the transmitted signal $f_{\mathbf{u}}(\mathbf{x})$ as nearly as possible, making use of the bandlimited hypothesis and the received data $r(\mathbf{x})$. The uncertainty principle can make stable signal recovery possible as long as $0<|Q||T|< \frac{2\pi}{|\mathbf{b}|}$.
That is to say, if there exists a linear operator $P$ and a constant $\Upsilon$ such that $\|f_{\mathbf{u}}-Pr\|_{L^2(\mathbb{R}^2,\mathbb{H})}\leq\Upsilon\|n\|_{L^2(\mathbb{R}^2,\mathbb{H})}$, where $\Upsilon\leq\left(1-\sqrt{\frac{|\mathbf{b}|}{2\pi}}\sqrt{|Q||T|}\right)^{-1}$. Then $f_{\mathbf{u}}$ can be stably reconstructed from $r$.
The application of the uncertainty principles for QWLCT in signal recovery is our main research direction in the future.

Moreover, the derived results in this paper can be extended to other transforms, such as Wigner Distribution (WD) and ambiguity function
(AF) \cite{40,41}.
\section{Conclusions}

The QWLCT is not only a linear transform, but also has similar properties as the WLCT \cite{8,14}.
In this paper, Pitt's inequality and Lieb inequality for the QWLCT are investigated and different forms of uncertainty principles associated with the QWLCT are proposed.
Firstly, we present some important properties of the QWLCT. These properties have been proved in \cite{37}. Secondly, based on the relationship between the QWLCT and
the QFT, the Pitt's inequality and Lieb inequality associated with the QWLCT are demonstrated. Thirdly,  the uncertainty
principles for the QWLCT such as Logarithmic uncertainty principle, Entropic uncertainty principle, Lieb
uncertainty principle and Donoho-Stark's uncertainty principle are obtained. Finally, we give a numerical example and a potential application in signal recovery by using uncertainty principle. The results can be viewed as a generalized form of other transforms uncertainty relations, such as the QLCT, QWFT and QFT.
 Further
investigations on this topic are now under investigation such as the left-sided QWLCT, the
right-sided QWLCT and the QWLCT directional uncertainty principle. Moreover, we will discuss the physical significance and engineering background of this paper in the further. They will be reported in a forthcoming paper.

\section*{Data availability statement}
Data sharing not applicable to this article as no datasets were generated or analysed during the current study.

\end{document}